\documentclass[12pt]{amsart}

\usepackage{latexsym}
\usepackage{amsthm}
\usepackage{amsmath}
\usepackage{amssymb}
\usepackage{amsbsy}
\usepackage{hyperref}
\numberwithin{equation}{section}

\theoremstyle{plain}
\newtheorem{thm}{Theorem}[section]
\newtheorem{lem}[thm]{Lemma}
\newtheorem{prop}[thm]{Proposition}
\newtheorem{cor}[thm]{Corollary}

\def \RR{{\mathcal R}}
\def \mm{{\mathfrak m}}
\def \reg{{\text{reg }}}

\def \deg{{\text{deg}}}
\def \ord{{\text{ord}}}

\def \dim{{\text{dim}}}
\def \Proj{{\text{Proj}}}
\def \PP{{\mathbb{P}}}
\def \ZZ{{\mathbb{Z}}}
\def \NN{{\mathbb{N}}}
\def \aa{{\bf{a}}}

\def \dd{{\bf{d}}}
\def \xx{{\bf{x}}}

\def \11{{\bf{1}}}

\title{Regularity defect stabilization of powers of an ideal} 
\author{David Berlekamp}
\address{Department of Mathematics, University of California, Berkeley, CA 94720-3840, USA}
\email{daffyd@math.berkeley.edu}
\date{\today}
\begin{document}

\begin{abstract}
When $I$ is an ideal of a standard graded algebra $S$ with homogeneous maximal ideal $\mm$, it is known by the work of several authors that the Castelnuovo-Mumford regularity of $I^m$ ultimately becomes a linear function $dm + e$ for $m \gg 0$.  

We give several constraints on the behavior of what may be termed the \emph{regularity defect} (the sequence $e_m = \reg I^m - dm$) in various cases.  When $I$ is $\mm$-primary we give a family of bounds on the first differences of the $e_m$, including an upper bound on the increasing part of the sequence; for example, we show that the $e_i$ cannot increase for $i \geq \dim(S)$.  When $I$ is a monomial ideal, we show that the $e_i$ become constant for $i \geq n(n-1)(d-1)$, where $n = \dim(S)$.  
\end{abstract}

\maketitle

\section{Introduction}

When $I$ is an ideal of a standard graded algebra $S$ with homogeneous maximal ideal $\mm$, it is known by the work of several authors that the Castelnuovo-Mumford regularity of $I^m$ ultimately becomes a linear function $dm + e$ for $m \gg 0$ (c.f. \cite{CHT99}, \cite{Kod00}, \cite{TW05}).  The leading coefficient $d$ is the asymptotic generating degree of $I$, i.e. the minimal number such that $I$ is integral over $I_{\leq d}$, where $I_{\leq d}$ denotes the ideal generated by the forms in $I$ of degree at most $d$.  The other coefficient, $e$, is more mysterious.  In the case where $I$ is $\mm$-primary and generated by general forms of a single degree, $I$ determines a finite morphism $\phi: \Proj(S) \to \PP^r$, and \cite{EH10} showed that $e+1$ is the maximum regularity of a fibre of $\phi$; following a conjecture of \cite{Ha10}, \cite{Char10} extended this result to the case in which $I$ is generated by arbitrary forms of a single degree in terms of the associated map from the blowup of $\Proj(S)$ along the subscheme defined by $I$.  This is quite a tantalizing phenomenon, but so far no similar interpretation has presented itself in a more general case.  

We define the \emph{regularity defect sequence} (or, shortly, the \emph{defect}) to be the sequence $e_m = \reg I^m - dm$, and the \emph{asymptotic} or \emph{stable defect} $e_{\infty}$ to be $e_m$ for $m \gg 0$..  Again in the case where $I$ is $\mm$-primary and generated by forms of a single degree, \cite{EH10} showed that the defect sequence is weakly decreasing, and \cite{EU10} gave an explicit bound on the power $m$ at which the defect sequence stabilizes, no more than the $(0,1)$-regularity of the Rees algebra $\RR(I)$ (see \cite{EU10}); they also explored some examples in the case where $I$ may not be generated in a single degree, though the situation there thus far eludes clear characterization.  M. Chardin has announced a related result giving a criterion for stabilization of the defect; specifically, that $e_{m+1}=e_m$ if $m$ is at least $\text{reg}_{k[V]}(\RR(I)_{e_{\infty},*})$, where $V \subseteq I_d$ is a subspace such that $V+\mm I_{d-1} = I_d$, and also
$$t\geq \min \left( \min\{ m | I^m \subseteq I_{\leq d}\}, \frac{\reg(I_{\leq d}) - d}{d+1} \right)$$
This nicely parallels the stabilization theorem of \cite{EU10} in the equigenerated case, in which the fact that $I = I_{\leq d}$ renders the latter requirement trivial.  
The bound therein is much like one of those of our Theorem~\ref{thm:bd} ensuring that the defect sequence is weakly decreasing.  In the general case, despite this progress, many properties of the regularity defect remain mysterious.  

In this note, we give several constraints on the behavior of the defect sequence.  In particular, we show that when $I$ is $\mm$-primary and $S$ has characteristic $0$, the defect sequence is weakly decreasing for $m$ larger than $(\reg I_{\leq d})/(d+b')$, where $d+b'$ is the degree of a minimal generator of $I$ of minimal degree strictly greater than $d$ (if no such generator exists, as is for instance the case when $I$ is generated in a single degree, the defect sequence is weakly decreasing, recovering a result of \cite{EH10}).  We observe that this bound is strictly smaller than the dimension of $S$.  We also show that all first differences of the defect sequence are bounded above by $b$, where $d+b$ is the maximal degree of a minimal generator of $I$, and give an interpolating set of increasingly powerful bounds on the first differences of the defect sequence kicking in at larger values of $m$.  These bounds relate to a question of \cite{EU10} as to whether the sequence of first differences is weakly decreasing; such is not the case once the defect sequence has begun to decrease (witness examples in Section~\ref{sec:ex}), but seems to be so for the increasing part.  All of this comprises Section~\ref{sec:diff} of this paper.  

In Section~\ref{sec:mon}, we turn to the case where $I$ is an $\mm$-primary monomial ideal.  In this case, we observe that if $n$ is the number of variables of $S$, then the minimal generators of $I$ include a regular sequence of pure powers of the variables, and $d$ is the maximum pure power to appear.  We are able to show that if $l \leq n$ is the number of variables appearing to the power $d$ as minimal generators of $I$, then the regularity defect sequence achieves its asymptotic value for all $m$ at least $l(n-1)(d-1)$; a somewhat sharper result requiring more apparatus to articulate is also presented.   

Finally, in Section~\ref{sec:ex} we turn to the computation of some specific examples exploring the precision of the bounds we have obtained.

\section{First differences of the regularity defect sequence}
\label{sec:diff}

Fix a polynomial ring $S = k[\xx] = k[x_1,\ldots,x_n]$.  For $\aa \in \NN^n$, let $\xx^{\aa} = x_1^{a_1}\cdots x_n^{a_n}$, and let $\mm^{\aa} = (x_1^{a_1}, \ldots, x_n^{a_n})$.  We also write $\11=(1,\ldots,1)$, so $\11\cdot \aa = \sum a_i$ is the total ($\ZZ$-)degree of $\xx^{\aa}$.  
Let $$\sigma(I) = \{\alpha \in S/I : \alpha \neq 0, \mm \alpha =0 \}$$  i.e., $\sigma(I)$ is the set of nonzero elements of the socle of $S/I$.  
In this situation, we define $\omega(I)\subset \sigma(I)$ to be the set of maximal total degree elements of $\sigma(I)$.  We also define $s(I)$ and $w(I)$ to be the sets of homogeneous elements of $S$ mapping to $\sigma(I)$ and $\omega(I)$ respectively under the quotient.  Observe that if $S/I$ is artinian and $f \in w(I)$, then $\deg(f) = \reg S/I = \reg I -1$.  We call $w(I)$ the \emph{witness set} of $I$, and a form $f \in w(I)$ a \emph{witness for (the regularity of)} $I$.  

\begin{thm}
\label{thm:1diff}
Let $I \subset S$ be any $\mm$-primary homogeneous ideal of asymptotic generating degree $d$, and let $\reg I^m = dm + e_m$.  Let $J \subset I$ be another $\mm$-primary homogeneous ideal, and let $c$ be the maximal degree of a minimal generator of $J$.  If $J \cap w(I^m) \neq \varnothing$, then $\reg I^m - \reg I^{m-1} \leq c$, and hence $$e_m - e_{m-1} \leq c-d.$$
\end{thm}

\begin{proof}
Let $f \in J\cap w(I^m)$.  By definition, as $f\in w(I^m)$ we have $f \notin I^m$ and $\deg(f) = \reg S/I^m = dm + e_m -1$.

As $f \in J$, write $f = \sum u_ig_i$, where the $g_i$ are minimal generators of $J$ and the $u_i$ are some forms, such that no summand can be discarded.  As $f \notin I^m$, not all of the $u_i$ lie in $I^{m-1}$; without loss of generality assume $u_1 \notin I^{m-1}$.  Then $\deg(u_1) \leq \reg S/I^{m-1}$, so $$\reg I^m = \deg(f) +1= \deg(u_1) + \deg(g_1) +1 \leq \reg I^{m-1} + c.$$  

Hence $$e_m - e_{m-1} = \reg I^m - (\reg I^{m-1} + d) \leq c-d.$$
\end{proof}

\begin{cor}
\label{cor:interp}
Let $J\subset I$ be $\mm$-primary homogeneous ideals as above, with $d$ the asymptotic generating degree of $I$, $\reg I^m = dm + e_m$, and $c$ the maximal degree of a minimal generator of $J$.  Whenever $\reg I^m > \reg J$, we have $$e_m - e_{m-1} \leq c-d.$$
\end{cor}

\begin{proof}
$\reg I^m > \reg J$ implies that $w(I^m) \subset J$; the result follows by Theorem~\ref{thm:1diff}.  
\end{proof}

Setting $J=I$ recovers the result of \cite{EU10} that the defect sequence is weakly decreasing in the equigenerated case, and also extends it to the case where all generators have degree at most $d$.  

Regularity is weakly increasing in powers (simply because $I^{m+1} \subset I^m$).  However, regularity often fails to strictly increase in products; for example, if $I,J$ are distinct ideals of the same regularity, it is not generally the case that the regularity of the product will be larger.  For example, if $n=4$, $I$ is generated by the squares of the variables, and $J$ is generated by squares of four generic linear forms, then $\reg I = \reg J = \reg IJ = 5$.  

This observation leads to some more general questions:  If $n \geq 4$ and $I,J$ are generated by distinct generic regular sequences of the same degree $d\geq 2$, is it the case that $\reg IJ = \reg I = \reg J$? Is the same true if $I$ is generated by pure powers of the variables and $J$ by pure powers of generic linear forms?  Experimental evidence indicates an affirmative answer, i.e. that a generic form $f$ of the appropriate degree satisfies $f \in w(I) \cap w(J)$ and $\mm f \in IJ$, so $\reg I = \reg J = \reg IJ$.  A heuristic justification for the significance of $n=4$ may be given by the following argument: Let $I$, $J$ be given by distinct maximal regular sequences of degree $d$, so $\reg I = \reg J = n(d-1)+1 =: q$.  Then as any form in $w(I) + w(J)$ cannot be in $IJ$, $\reg IJ = \reg I$ if and only if $\mm^q \subset IJ$.  Observing that $(IJ)_q = I_dJ_{q-d}$, this requires $\dim_k(\mm^q) \leq n\dim_k(J_{q-d})$.  As $S/J$ is Gorenstein with socle in degree $q-1$, $\dim_k((S/J)_{q-d}) = \dim_k((S/J)_{d-1})= \dim_k(\mm^{d-1})$, and $\dim_k(J_{q-d}) = \dim_k(\mm^{q-d}) - \dim_k((S/J)_{q-d})$; hence the requirement on $n$ and $d$ comes down to $${{nd}\choose{n-1}} \leq n\left[ {{(n-1)d}\choose{n-1}} - {{d+n-2}\choose{n-1}}\right]$$
and this is true for $n\geq 4$, $d\geq 2$ (but not for $n=2$ nor $n=3$ for any $d$).  Not many lower bounds on the regularity of products are known, and an answer to this question in general would be a fascinating result.

Nonetheless, with a bit more work we can show that the regularity is in fact strictly increasing in powers (at least in characteristic $0$).  

\begin{prop}
\label{prop:regup}
Let $I_1,\ldots, I_m$ be homogeneous ideals of $S$, with a form $f \in S$ such that $$\mm f \in \prod_iI_i$$ and assume the characteristic of $S$ does not divide $n +\deg(f)$.  Then $$f \in \sum_i \prod_{j\neq i} I_j.$$
\end{prop}

\begin{proof}
Let $\mm = (x_1, \ldots, x_n)$.  By hypothesis, there are finitely many forms $g_{ik} \in I_i$ such that for $1\leq k \leq n$ we have $x_k f = \sum_k \prod_i g_{ik}$.  Then $$\sum_k [x_k f]_{x_k} = \sum_{k,i} [g_{ik}]_{x_k}\prod_{j \neq i} g_{jk} \in \sum_i \prod_{j\neq i} I_j$$
but also $$\sum_k [x_k f]_{x_k} = nf + \sum_k x_k[f]_{x_k} = (n + \deg(f))f$$
and thus $f \in \sum_i \prod_{j\neq i} I_j$.  
\end{proof}

Henceforth we assume that $S$ has characteristic $0$ so as to apply Proposition \ref{prop:regup}.  
\begin{prop}
\label{prop:sdisj}
Let $I$ be a homogeneous ideal of  $S$, and let $m > m' > 0$.  Then $s(I^m) \subset I^{m'}$.  
In particular, $s(I^m) \cap s(I^{m'}) = \varnothing$.  
\end{prop}
\begin{proof}
Let $I_i = I$, $i = 1,\ldots, m$, and apply Proposition~\ref{prop:regup} to see that any $f \in s(I^{m})$ satisfies $f \in I^{m-1} \subset I^{m'}$.  For the last, observe that any homogeneous ideal $J$ is disjoint from $s(J)$.  
\end{proof}

\begin{lem}
\label{lem:wit}
Let $I,J$ be $\mm$-primary homogeneous ideals.  If there is a witness for $I$ that is neither in $J$ nor a witness for $J$, then $\reg I < \reg J$.  
\end{lem}
\begin{proof}
Let $f \in w(I) \backslash (J \cup w(J))$.  Then $f \notin J$, and $\mm f \not\subset J$ as $f \notin w(J)$.  So let $g \in \mm f \backslash J$, and observe that $\reg I = \deg(f) + 1 = \deg(g) < \reg J$.  
\end{proof}

\begin{cor}
Let $I$ be an $\mm$-primary homogeneous ideal.  Then for all $m>0$, $\reg I^{m+1} > \reg I^m$.
\end{cor}
\begin{proof}
$\reg I^m = \deg f$, where $f$ is any element of $w(I^m)$; since $w(J) \subset s(J)$ for any $J$, $f \notin w(I^{m+1})$ by Proposition~\ref{prop:sdisj}.  Now use Lemma~\ref{lem:wit}.
\end{proof}

We can use these tools to get bounds on the increasing part of the regularity defect sequence.  

\begin{thm}  
\label{thm:bd}
Let $J\subset I$ be $\mm$-primary homogeneous ideals as above, with $d$ the asymptotic generating degree of $I$, $\reg I^m = dm + e_m$, and $c, c'$ respectively the maximal and minimal degrees of minimal generators of $J$.  Let $d'$ be the minimal degree of an element of $g(I) \backslash J$, where $g(I)$ denotes the set of minimal generators of $I$.  Whenever $m > \min\{(\reg J)/d,(\reg J)/d' + \max\{1 - c'/d',0\}\}$
we have $\reg I^m > \reg J$, and hence by Corollary~\ref{cor:interp}, $$e_m - e_{m-1} \leq c-d.$$

This holds for each $m$ such that $w(J) \not\subset s(I'^m) + w(I'^{m-1}J)$, where $I'$ is the ideal generated by $g(I) \backslash J$. 

\end{thm}

\begin{proof}

When $m > (\reg J)/d$, manifestly $\reg I^m = dm + e_m > \reg J$.  

Otherwise, assume $d'm> (\reg J)$.  Then $I^m \subset J$, so $\reg I^m \geq \reg J$.  Now $I^m = (J + I')^m = \sum_{i=0}^m J^iI'^{m-i}$.  All but the last two terms are contained in $J^2$, and since $\reg J^2 > \reg J$, $\reg I^m$ can only equal $\reg J$ if $w(J) \subset s(I'^m) + w(I'^{m-1}J)$.  The minimum degree of an element of $s(I'^m)$ is $d'm-1$, while the minimum degree of an element of $s(I'^{m-1}J)$ is $d'(m-1) +c' -1$;  if $d'm > \reg J + d' - c'$ as well, both are larger than $\reg J-1$, the degree of any element of $w(J)$.  

\end{proof}

In the case where $I$ is equigenerated, \cite{EU10} showed that $e_{m+1} \leq e_m$ for all $m > 0$.  We extend this result as follows.  

\begin{cor}
\label{cor:dec}
Let $I \subset S$ be an $\mm$-primary homogeneous ideal of asymptotic generating degree $d$, and let $\reg I^m = dm + e_m$.  Let $d+b$ be the maximal degree of a minimal generator of $I$, let $d+b'$ be the minimal degree strictly greater than $d$ of a minimal generator of $I$ ($b' = \infty$ if no such generator exists), and let $I_{\leq d}$ be the ideal generated by the elements of $I$ of degree at most $d$.  Then 
\begin{enumerate}
\item $e_{m+1} \leq e_m + b$ for all $m > 0$.
\item $e_{m+1} \leq e_m$ for all $m > (\reg I_{\leq d})/(d+b')$.
\end{enumerate}

Note the latter case holds whenever $m \geq n$.  

\end{cor}
\begin{proof}
Let $u$ be a minimal generator of $I$ of degree $d+b$.  
Let $J = I_{\leq d}$. The minimal degree of an element of $I \backslash J$ is $d+b'$, which yields the first statement by Theorem~\ref{thm:bd}.

Setting $J=I$ yields the second statement.

The largest possible value of $\reg I_{\leq d}$ is $n(d-1) + 1$, attained only when $I_{\leq d}$ is generated by a regular sequence of degree $d$.  Combining this with (2) makes the final observation manifest.

\end{proof}

\section{Stabilization in the monomial case}
\label{sec:mon}

We continue with analysis of the monomial case, in which we can give a bound on the stabilization of the defect sequence that is in some cases sharp.  The basic idea will be, firstly, to show that for sufficiently large $m$, one can always find a monomial witness $a \in w(I^m)$ that has a large order with respect to some one of the variables; and secondly, to show that for sufficiently large $m$, the monomials in $I^{m+1}$ of large order with respect to any one of the variables are simply a shift of those in $I^m$.  We will do this by examining monomial exponent vectors modulo the lattice generated by the exponents of the pure powers appearing as minimal generators of $I$.  

Any $\mm$-primary monomial ideal $I$ of $S$ contains pure powers of each variable among its minimal generators, the largest of which is $d$, the asymptotic generating degree of $I$.  Therefore, renaming the $x_i$ appropriately, such an ideal may be written as $$I = (x^d, y_1^d, \ldots, y_{l-1}^d, z_1^{d_1}, z_2^{d_2}, \ldots, z_k^{d_k}, h_1, \ldots, h_r)$$ where $d > d_1 \geq d_2 \geq \cdots \geq d_k$, $l+k = n$, and $h_1, \ldots, h_r$ are the non-pure minimal generators of $I$.  

\begin{thm}
\label{thm:simplebd}
Let $I$ be an $\mm$-primary monomial ideal of $S$, with asymptotic generating degree $d$, and number of pure generators of degree $d$ equal to $l$ as above.  Then $e_{m+1} = e_m$ for all $$m > \max\{n-1, (n-1)[l(d-1) - 1]\}.$$  
\end{thm}

To prove this, we require some additional notation, which will also allow a stronger statement.  

Given a homogeneous ideal $J$ and a form $f \in S$, let $\ord_J(f)$ denote the $J$-order of $f$, i.e. the unique natural number $t$ such that $f \in J^t \backslash J^{t+1}$.   If $I$ is as above, let $Y = (y_1^d, \ldots, y_{l-1}^d)$, $y = (y_1, \ldots, y_{l-1})$, $Z = (z_1^{d_1}, \ldots, z_k^{d_k})$, and $z = (z_1, \ldots, z_n)$.  If $a$ is a monomial of $S$, we factor $a$ as $$a = x^{m_xd + o_x}y_1^{m_{y_1}d+o_{y_1}}\cdots y_l^{m_{y_l}d+o_{y_{l-1}}}z_1^{m_{z_1}d_1+ o_{z_1}}\cdots z_k^{m_{z_k}d_k+o_{z_k}}$$ where the $m_{x_i}$ are as large as possible, so for instance $\ord_Y(a) = \sum_i m_{y_i}$.  We define the \emph{reduction of $a$} to be $\bar{a} = x^{o_x}y_1^{o_{y_1}}\cdots z_k^{o_{z_k}}$, and $a_x = x^{m_xd}\bar{a}$.  Given an natural number $m$, we write $\bar{m} = \bar{m}(m,a) = m - m_x - \ord_Y(a) - \ord_Z(a)$.  Finally, let $\mu_m = \min_{a \in w(I^m)}\bar{m}(a)$.

The more refined theorem we will prove is the following.  

\begin{thm}
\label{thm:inc}
Let $I$ be an $\mm$-primary monomial ideal of $S$, with asymptotic generating degree $d$, and number of pure generators of degree $d$ equal to $l$.  Then $e_{m+1} \geq e_m$ if $m > (n-1)(d-2) + (\mu_m - 1)(l-1)(d-1)$.  
\end{thm}

Given that $\mu_m \leq n-1$ by Proposition~\ref{prop:mbarbd}, this combines with Corollary~\ref{cor:dec} to give Theorem~\ref{thm:simplebd}.  The proof will require a few technical results.  

\begin{prop}
\label{prop:mbarbd}
If $a$ is a monomial in $w(I^m)$, then 
\begin{enumerate} 
\item{$1\leq \bar{m} \leq n-1$}
\item{$\bar{m}d -1 \leq \deg(\bar{a}) < l(d-1) + \sum_{i=1}^k (d_i-1) \leq n(d-1) - k$}
\item{$\ord_Z(a) \leq \deg(\bar{a}) - \bar{m}d + 1 \leq (n-1)(d-1)-1$}
\end{enumerate}
\end{prop}

\begin{proof}
The monomial exponent vector of $\bar{a}$ is the remainder of that of $a$ when reduced by the lattice of exponents of pure generators of $I$, so in particular $a \notin I^m$ implies that $\bar{a} \notin I^{\bar{m}}$.  This shows that $\bar{m} \geq 1$.  Also, each exponent of a variable in $\bar{a}$ is less than the corresponding pure power of that variable appearing as a minimal generator of $I$, giving the upper bound in (2).  The upper bound in (3) follows.  

As $e_m \geq 0$, the hypothesis of the proposition ensures that $$md -1 \leq \deg(a) = \deg(\bar{a}) + m_xd + \sum_{i=1}^{l-1}{m_{y_i}d} + \sum_{i=1}^k{m_{z_i}d_i}$$  Subtracting $(m-\bar{m})d$ from both sides yields $$\bar{m}d-1 \leq \deg(\bar{a}) - \sum_{i=1}^k{m_{z_i}(d-d_i)}$$
and as all $d_i < d$ this gives the lower bound in (2).  Combined with the observation that $\sum{m_{z_i}} = \ord_Z(a)$, this also gives the lower bound in (3).  Finally, dividing (2) by $d$ gives $\bar{m} < n\left(\frac{d-1}{d}\right) - \left(\frac{k+1}{d}\right) < n$.    
\end{proof}

Next, we will impose constraints on $\bar{a}$ that guarantee stabilization of the defect.  The first lemma in this direction shows that the defect is nondecreasing if appropriate shifts of a witness fail to lie in corresponding powers of $I$.  

\begin{lem}
\label{lem:common}
Suppose $a\in w(I^m)$.  
If for all $q\geq 0$ we have $x^{dq}a_x \notin I^{m_x+\bar{m} + q}$, then $\ord_Z(a)=0$, and $e_{m'} \geq e_{m}$ for all $m' \geq m$.
\end{lem}
\begin{proof}
Set $q = \ord_Y(a) + \ord_Z(a) +q'$; then the hypothesis ensures $x^{dq}a_x \notin I^{m+q'}$, and $$\deg(x^{dq}a_x) = \deg(a) + dq' + \sum_{i=1}^k{m_{z_i}(d-d_i)}$$  Setting $q' = 0$ gives a monomial $x^{dq}a_x \notin I^m$, whose degree can be no larger than that of $a \in w(I^m)$.  As all $d_i <d$, this shows that all $m_{z_i}=0$, so their sum $\ord_Z(a) = 0$ also.  We now have a monomial not in $I^{m+q'}$ of degree $\deg(a) + dq'$ for all $q'\geq 0$, showing that $e_{m+q'} \geq  e_m$ for all $q'\geq 0$.  
\end{proof}

The next lemma shows that if the reduction of a witness has sufficiently large order with respect to $x$, then we can find a witness with the same reduction and order $0$ with respect to $Y+Z$, shifts of which will remain outside higher powers of $I$.  

\begin{lem}
\label{lem:abar2}
If $a \in w(I^m)$ and $\ord_x(\bar{a}) > \bar{m}d - (m_x+\bar{m}+1)$, then $\ord_Z(a) = 0$, and $e_{m'} \geq e_m$ for all $m' \geq m$.
\end{lem}
\begin{proof}
Consider the set of monomials $B_t = \{\text{monomials } b' \in I^t \backslash x^dI^{t-1}\}$.  
A monomial $b' \in B_t$  is a product of $t$ monomials in $I$, none of which is divisible by $x^d$, and thus satisfies $\ord_x(b') \leq t(d-1)$.  Thus if $b \notin I^s$ is a monomial such that $x^db \in I^{s+1}$, we have $x^db \in B_{s+1}$, and so $\ord_x(b) \leq (s+1)(d-1) - d = s(d-1) - 1$.  
Therefore, if $b \notin I^s$ is a monomial such that $\ord_x(b) > s(d-1) -1$, it follows that $x^db \notin I^{s+1}$.  Since $\ord_x(x^db) = \ord_x(b) + d > s(d-1) -1 + d > (s+1)(d-1) -1$, the same argument applies to $x^db$ to show that $x^{2d}b \notin I^{s+2}$, and inductively yields $x^{dq}b \notin I^{s+q}$ for all $q\geq 0$.  Again, $a\in w(I^m)$ implies that $a_x \notin I^{m_x+\bar{m}}$, and $\ord_x(a_x) = \ord_x(\bar{a}) + m_xd$, so under the hypothesis of the lemma, $\ord_x(a_x) > (m_x + \bar{m})(d-1)-1$.  Setting $b = a_x$ and $s = m_x+\bar{m}$ shows that $x^{dq}a_x \notin I^{m_x+\bar{m} + q}$ for all $q \geq 0$; again, apply Lemma \ref{lem:common}.
\end{proof}

With this, we can complete the proof of Theorem~\ref{thm:inc}.  

\begin{proof}
Suppose $e_{m+1} < e_m$.  Observe that Lemma \ref{lem:abar2} applies when $x$ is replaced by any $y_i$.  Relabel $x$ as $y_l$ (and include $y_l$ in $y$, and $y_l^d$ in $Y$ accordingly).  Then, for all $a \in w(I^m)$ and all $i$ with $1 \leq i \leq l$ it must be that $\ord_{y_i}(\bar{a}) \leq \bar{m}d - (m_{y_i} + \bar{m} + 1)$.  Summing over $i$, we have $$\ord_y(\bar{a}) \leq \bar{m}l(d-1) - l - \ord_Y(a)$$ and remembering that $\ord_Y(a) + \ord_Z(a) + \bar{m} = m$ we have $$\ord_y(\bar{a}) \leq \bar{m}l(d-1) - l - m + \bar{m} + \ord_Z(a)$$
Applying Proposition \ref{prop:mbarbd} (3), this implies $$\ord_y(\bar{a}) \leq \bar{m}(l-1)(d-1)-l-m - \deg(\bar{a}).$$ 
Adding $m-\ord_y(\bar{a})$ to each side of this inequality gives $$m \leq \bar{m}(l-1)(d-1) - (l-1) + \ord_z(\bar{a})$$ and combining this with the trivial bound on $\ord_z(\bar{a})$ yields $$m \leq \bar{m}(l-1)(d-1) - (n-1) + \sum_{i=1}^kd_i.$$
Finally, using that each $d_i \leq d-1$ gives $$m \leq (\bar{m}-1)(l-1)(d-1) +(n-1)(d-2)$$
and the theorem follows. 
\end{proof}

Since we know that $e_{m+1} \leq e_m$ for all $m \geq n$ by Corollary~\ref{cor:dec}, and that all $\mu_m <n$ by Proposition~\ref{prop:mbarbd}, we have Theorem~\ref{thm:simplebd}.

\section{Examples}
\label{sec:ex}

The following examples illustrate some subtleties of the above.  

\begin{subsection}{Necessity of the hypotheses of Theorem~\ref{thm:bd}}
Let $n=3$, $$J= \sum_{1\leq i,j,k \leq 3} \langle x_i^8, x_i^7x_j^2, x_i^7x_jx_k \rangle,$$ and $I = J + \mm^{10}$.  Then $I^2 \subset J \subset I$, $S/J$ is generated as a vector space by $x_1^6x_2^6x_3^6$, and $\omega(I^2)$ is generated by the image of $w(J)$ and $x_i^7x_j^6x_k^5$ for $(ijk) \in S_3$.  Note that $\reg I = 10$, $\reg J = \reg I^2 = 19$, and $w(J) \subsetneq = w(I'J) = w(I^2)$; incidentally, this affords a monomial example of a situation in which multiplying a given ideal (namely $J$) by a nontrivial ideal (namely $I'$) fails to induce the regularity to increase.  In the notation of Theorem~\ref{thm:bd}, we have $I'$ generated by  the monomials of degree $10$ not in $J$, $d=c'=8$, $c = 9$, $d'=10$; so we should have $e_m - e_{m-1} \leq 1$ for $m < 3$, while $e_m - e_{m-1} \leq 0$ for $m \geq 3$.  Indeed, $e_1 = 2$ and $e_2 = 3$, so the defect sequence in this case is
$$e=(2,3,3,\ldots)$$
\end{subsection}

\begin{subsection}{Initially increasing defect sequences}
\label{ex:inc}
Let $\dd = d\11$, and let $I=\mm^{\dd} + \mm^{d+b}$ (where $b \leq n(d - 1) -d$).  
As long as $$m(d+b) \leq \reg I_{\leq d} = \reg \mm^{\dd} = \11\cdot (\dd - \11) +1= n(d-1)+1,$$ $\omega(I^m)$ is generated by the image of $g(\mm^{m(d+b)-1}) \backslash \mm^{\dd})$, and so $\reg I^m = m(d+b)$.  
Hence for
$m \leq m_0 := \lfloor \frac{n(d-1) + 1}{d+b} \rfloor$, $e_m = mb$, so $e_m -e_{m-1} = b$.  For $m > m_0$, we have $e_m = e_{m+1}$.  Finally, $$e_{m_0+1} - e_{m_0} = \delta:= \max(n(d-1) +1 - m_0(d +b) -d,0) < b$$  Observe that this example exactly achieves the bound on the length of the increasing part of the sequence, as well as the bound on the first differences of the sequence, given in Corollary~\ref{cor:dec}.  Also observe that $\mu_m = m_0$ for large $m$, so if $b$ is small and $d$ large, $\mu_m = m-1$, although in this case the defect sequence never decreases.  

Here, the defect sequence is 
$$e = (b, 2b, 3b, \ldots, m_0 b, m_0 b + \delta, m_0 b + \delta, \ldots)$$
\end{subsection}

\begin{subsection}{Slowly decreasing defect sequences}
Let $$I= x^{d-1}\mm + (z_1^{d-1},\ldots, z_k^{d-1})$$  
Here $l=1$ and $k=n-1$.  $w(I)$ is generated by $x^{d-2} \left(\prod_i z_i\right)^{d-2}$; in fact, for $m <(n-1)(d-2)$, $w(I^m)$ is generated by $\{ x^{jd-j-1}\left(\prod_{i\in \Lambda}z_i^{d-1}\right) \left(\prod_i z_i\right)^{d-2}\}$ where $j \in[m]$, and $\Lambda$ is a multiset on $[k]$ with $|\Lambda|=m-j$.  Hence for $m < (n-1)(d-2)$, $e_{m+1} - e_m = -1$, while for larger $m$, $e_{m+1} - e_m = 0$.  Observe that this example exactly achieves the bound on the length of the decreasing part of the sequence given by Theorem~\ref{thm:inc} (as in this case $l-1 = 0$).  Observe also that if the $z_i$ variables were to instead appear to the power $d$, the defect sequence would stabilize immediately to $e_1 = e_{\infty} = n(d-1) -1$.  

Thus, the defect here is 
$$e = ((n-1)(d-2), (n-1)(d-2)-1, \ldots, 1,0,0,\ldots)$$

\end{subsection}

Careful consideration of such behavior leads to the following interesting characterization of $\mm$-primary monomial ideals with stable defect equal to $0$.  This extends the corresponding result in the equigenerated case obtained in \cite{EU10}.  

\begin{prop}
\label{prop:inftyeqz}
Let $I$ be an $\mm$-primary monomial ideal of $S$, with asymptotic generating degree $d$.  Then $e_{\infty}=0$ if and only if for each pure minimal generator $x^d$ of degree $d$, $x^{d-1}\mm \subset I$.  
\end{prop}
\begin{proof}
If $x^{d-1}\mm \not \subset I$, let $a \in (x^{d-1}\mm) \backslash I$.  Then $x^{dq}a \notin I^{q+1}$ for all $q \geq 0$, so $e_q \geq \deg(x^{dq}a) - ((q+1)d-1) = 1$ for all $q \geq 0$ and $e_{\infty} \geq 1$.  

On the other hand, if $x^{d-1}\mm \subset I$ for each pure minimal generator $x^d$ of degree $d$, then for large $q$ every monomial in $\mm^{dq}$ outside of $Y + Z$ will lie in $I^q$.  By the proofs of Lemma~\ref{lem:abar2} and Theorem~\ref{thm:inc}, for sufficiently large $q$ there exists a witness $a \in w(I^q)$ lying outside of $Y+Z$ relative to some pure minimal generator $x^d$ of degree $d$.  This witness has degree at most $dq-1$, so $e_q = 0$.  Hence $e_{\infty} = 0$ also.  
\end{proof}

\begin{subsection}{Initially increasing then later decreasing defect sequences}
Let 
$$I = \sum_{i=1}^n (x_i^{d-1}\mm) + \mm^{d+b}$$
where $n>2$.  When $d$ is large and $b$ is small, the defect sequence must initially increase according to $b$ with each power of $I$, as in Example~\ref{ex:inc}; however, by Proposition~\ref{prop:inftyeqz}, $e_{\infty} = 0$.

Example 2.4 of \cite{EU10} is the case $n=4$, $d=5$, $b=1$, where $$e=(1,2,2,1,1,1,1,1,0,0,\ldots)$$
while for $n=4$, $d=5$, $b=2$ we have $$e=(2,3,2,2,2,2,2,1,0,0,\ldots)$$ 
and for $n=4$, $d=6$, $b=2$ the defect sequence begins with $(2,4,4,3,\ldots)$, but $e_{\infty}=e_{12}=0$.  

\end{subsection}

\section*{Acknowledgments}
Many thanks to Adam Boocher for the idea of Propsition~\ref{prop:regup}, and to David Eisenbud for general and indispensable guidance.

\nocite{Co06, EHU06, Chan97}

\bibliography{regdef}
\bibliographystyle{math}

\end{document}